\documentclass{amsart}

\usepackage{amsmath}
\usepackage{amssymb}
\usepackage[all]{xy}
\usepackage{picinpar}

\def\eu{\mathfrak}
\def\ma{\mathbb}

\def\lam#1{k(\Lambda_#1)}
\def\g#1#2{#1_{{\eu {ge}}#2}}

\def\cic#1{{\ma Q}(\zeta_{#1})}
\def\G#1{\big(R_T/(#1)\big)^{\ast}}
\def\S#1{S_{\infty}(#1)}
\def\p{{\eu p}_{\infty}}
\def\raiz#1{\sqrt[l]{P_{#1}}}
\def\raizm#1{\sqrt[l]{(-1)^{\deg P_{#1}}P_{#1}}}
\def\f{{\ma F}_q^{\ast}}
\def\Witt#1{\stackrel{_{\bullet}}{#1}}
\def\vWitt#1{\big(#1^{(1)},\ldots,#1^{(n)}\big)}
\def\ne{({\ma F}_q^{\ast})^l}

\newcommand{\Gal}{\operatorname{Gal}}

\newcommand{\Hom}{\operatorname{Hom}}
\newcommand{\menos}{\operatornamewithlimits{--}}

\newcounter{bean}

\def\l{
\begin{list}
{\rm{(\alph{bean}).-}}{\usecounter{bean}
\setlength{\labelwidth}{0.8in}
\setlength{\labelsep}{0.3cm}
\setlength{\leftmargin}{1cm}}}

\numberwithin{equation}{section}
\newtheorem{theorem}{Theorem}[section]
\newtheorem{proposition}[theorem]{Proposition}
\newtheorem{lemma}[theorem]{Lemma}
\newtheorem{example}[theorem]{Example}
\newtheorem{remark}[theorem]{Remark}
\newtheorem{definition}[theorem]{Definition}

\title[Genus Fields in Congruence Function Fields]
{Genus Fields of Abelian Extensions of
Congruence Rational Function Fields}
 
\author[M. Maldonado]{Myriam Maldonado--Ram\'irez}
\address{Departamento de Matem\'aticas\\
Escuela Superior de F\'isica y Matem\'aticas del I.P.N.}
\email{rosalia@esfm.ipn.mx}

\author[M. Rzedowski]{Martha Rzedowski--Calder\'on}
\address{Departamento de Control Autom\'atico\\
Centro de Investigaci\'on y de Estudios Avanzados del I.P.N.}
\email{mrzedowski@ctrl.cinvestav.mx}

\author[G. Villa]
{Gabriel Villa--Salvador}
\address{
Departamento de Control Autom\'atico\\
Centro de Investigaci\'on y de Estudios Avanzados del I.P.N.}
\email{gvilla@ctrl.cinvestav.mx}



\date{August 15, 2014}

\begin{document}

\begin{abstract}
In the published version of this paper [Finite Fields and Their
Applications {\bf 20} (2013) 40--54], there is an error in the proof of
Theorem 4.2 of the paper. Here
we correct the error and give the right statments for Theorems
4.2, 4.5 and 5.2

We give a construction of genus fields for congruence function
fields. First we consider the cyclotomic function field case
following the ideas of Leopoldt and
then the general case.
As applications we give explicitly the genus fields of Kummer, Artin--Schreier
and cyclic $p$--extensions. Kummer extensions were
obtained previously by G. Peng and Artin--Schreier extensions
were obtained by S. Hu and Y. Li.
\end{abstract}

\maketitle

\section{Introduction}\label{S1}

The concept of genus field goes back to Gauss \cite{Gau1801} in the
context of binary quadratic forms. For any finite extension $K/{\ma Q}$, 
the genus field is defined as the maximal unramified extension $K_{{\eu {ge}}}$
of $K$ such that $K_{{\eu {ge}}}$ is the composite of $K$ and an abelian extension
$k^{\ast}$ of ${\ma Q}$: $K_{{\eu {ge}}}=Kk^{\ast}$. This definition
is due to Fr\"ohlich \cite{Fro83}.
If $K_H$ denotes the Hilbert
class field of $K$, $K\subseteq K_{{\eu {ge}}}\subseteq K_H$. Originally the definition of
genus field was given for a quadratic extension of ${\ma Q}$. 
We have that for a quadratic number field $K$, the Galois group of
$K_{{\eu {ge}}}/K$ is isomorphic to the maximal subgroup of exponent $2$ of the
ideal class group of $K$. Gauss
in fact proved that if $t$ is the number of different positive finite
rational primes dividing the discriminant $\delta_K$ of a quadratic
number field $K$, then the $2$--rank of the class group of $K$
is $2^{t-2}$ if $\delta_K>0$ and there exists a prime $p\equiv
3\bmod 4$ dividing $\delta_K$ and $2^{t-1}$ otherwise.

H.W. Leopoldt \cite{Leo53} determined the genus field $K_{{\eu {ge}}}$ of an
abelian extension $K$ of ${\ma Q}$ using Dirichlet characters,
generalizing the work of H. Hasse \cite{Has51} who introduced
genus theory for quadratic number fields.

M. Ishida determined the genus field $K_{{\eu {ge}}}$ of any finite extension
of ${\ma Q}$ \cite{Ish76}. X. Zhang \cite{Xia85} gave a simple expression
of $K_{{\eu {ge}}}$ for any abelian extension $K$ of ${\ma Q}$ using Hilbert
ramification theory.

For function fields, the notion of Hilbert class field has no proper
analogue since the maximal abelian extension of any
congruence function field $K/{\ma F}_q$ contains $K_m:=
K{\ma F}_{q^m}$ for all positive integers $m$
 and therefore the maximal unramified abelian
extension of $K$ is of infinite degree over $K$.

M. Rosen \cite{Ros87} gave a definition of an analogue of the
Hilbert class field of $K$ and a fixed finite nonempty set $S_{\infty}$ of
prime divisors of $K$. Using this definition, a proper concept
of genus field can be given along the lines of the classical case.
R. Clement \cite{Cle92} considered a cyclic extension of $k:=
{\ma F}_q(T)$ of degree a prime number $l$ dividing $q-1$ and
found the genus field using class field theory. 
Later, S. Bae and J.K. Koo \cite{BaeKoo96}
generalized the results of Clement following the methods of
Fr\"ohlich \cite{Fro83}.

G. Peng \cite{Pen2003} explicitly described the genus theory for
Kummer function fields. Recently S. Hu and Y. Li \cite{HuLi2010}
explicitly described the ambiguous ideal classes and the genus field of
an Artin--Schreier extension of a rational congruence function
field.

In this paper we develop an analogue of
Leopoldt's genus theory for congruence function
fields. We give a description of the genus field $K_{{\eu {ge}}}$ of a finite
abelian extension of a rational congruence function field by 
means of the group of Dirichlet characters for cyclotomic
function fields. Here we consider the Hilbert class field $K_H$ of a
 function field $K$ using the construction of Rosen for $S_{\infty}=
\{{\eu p}_{\infty}\}$, where ${\eu p}_{\infty}$ is the pole
divisor of $T$ in the rational function field $k={\ma F}_q(T)$.

More precisely, let $K$ be a finite abelian extension of $k$.
Then if $K$ is contained in a cyclotomic extension, we find
that $K_{{\eu {ge}}}$ is also contained in a cyclotomic extension and we
find the group of characters associated to $K_{{\eu {ge}}}$. If $K$ is not
contained in a cyclotomic extension and $\p$ is tamely ramified,
 we consider a
suitable extension of constants of $K$ and then proceed
as before to find $K_{{\eu {ge}}}$. Finally, if $\p$ is wildly ramified we
consider the cyclotomic extension where $\p$ is totally and
wildly ramified and proceed similarly to the previous cases.

We apply our results to Kummer and to
Artin--Schreier extensions of $k$ and we give new proofs
of the results of Peng
and of Hu and Li.
At the end, we show that our construction also works to
find explicitly the genus field of an arbitrary
finite cyclic $p$--extension of $k$
given by a Witt vector.

\section{The classical case}\label{S2}

Let $K$ be a number field, that is, a finite extension of ${\ma Q}$.
Let $K_H$ be the Hilbert class field of $K$, that is, $K_H$ is the
maximal abelian unramified extension of $K$. Then the
genus field $K_{{\eu {ge}}}$ of $K$ is the maximal extension of $K$ 
contained in $K_H$ that is the composite of $K$
and an abelian extension $k^{\ast}$ of ${\ma Q}$. Equivalently,
$K_{{\eu {ge}}}=K k^{\ast} \subseteq K_H$ with $k^{\ast}$ the maximal
abelian extension of ${\ma Q}$ contained in $K_H$.

First we recall genus theory in the abelian case for number fields
\cite{Leo53}. In this case $K_{{\eu {ge}}}$ is the maximal extension
of $K$ contained in $K_H$ such that $K_{{\eu {ge}}}/{\ma Q}$ is 
abelian. So, in this section we consider $K/{\ma Q}$ an abelian
extension. By the Kronecker--Weber Theorem there exists
$n\in{\ma N}$ such that $K\subseteq \cic n$, where $\zeta_n$ denotes
a primitive $n$--th root of unity. Let $X$ be the group of 
Dirichlet characters associated to $K$. That is,
$X$ is a subgroup of the dual of $\Gal(\cic n/{\ma Q})\cong
U_n:=\big({\ma Z}/n{\ma Z}\big)^{\ast}$; then $X\subseteq \hat{U}_n$
and $K$ is the subfield of $\cic n$ fixed by $\cap_{\chi\in X}
\ker \chi$.

Let $n = p_1^{\alpha_1} \cdots p_r^{\alpha_r}$ be the factorization of $n$ as a product of prime powers. For any character $\chi$ let $\chi_{p_i} =  \chi \circ \varphi_i$ 
\begin{gather*}
\xymatrix{
U_n\ar[r]^{\chi}&{\ma C}^*\\
U_{p_i^{\alpha_i}}\ar[u]^{\varphi_i}\ar[ru]_{\chi_{p_i}}&}
\intertext{where  $\varphi_i = \Phi^{-1} \circ g_{p_i}$, with}
\begin{array}{rcl}
\Phi\colon U_n &\to& \prod\limits_{j=1}^{r}U_{p_j^{\alpha_j}}\\
a \bmod n &\mapsto& (a \bmod {p_j}^{\alpha_j})_j
\end{array}
\text{\ and\ }
\begin{array}{rcl}
g_{p_i}\colon U_{p_i^{\alpha_i}} &
\to& \prod\limits_{j=1}^{r}U_{p_j^{\alpha_j}}\\
a \bmod {p_i}^{\alpha_i} &\mapsto& (1, \ldots, a\bmod {p_i}^{\alpha_i}, \ldots, 1).
\end{array}
\end{gather*}

The character $\chi_{p_i}$ has conductor
$p_i^{\beta_i}$ for some $\beta_i\in{\ma N}$, $1\leq i\leq r$.
For any rational prime $p\notin\{p_1,\ldots,p_r\}$, $\chi_p=1$.
Let $p$ be a rational prime and define $X_p:=\{\chi_p\mid
\chi\in X\}$. Then we have $|X_p|=e_p$ is the ramification index
of $p$ in $K$. Thus,

\begin{theorem}[Leopoldt \cite{Leo53}]\label{T2.1} Let $K$
be an abelian extension of ${\ma Q}$ and let $X$ be the 
group of Dirichlet characters associated to $K$. 
Let $J$ be the maximal abelian extension of ${\ma Q}$ containing
$K$ such that $J/K$ is unramified at
every finite rational prime. Let $Y$ be the group of Dirichlet
characters associated to $J$. Then
$Y=\prod_{p\in{\mathcal P}}X_p$, where the product runs through
the set of rational primes ${\mathcal P}$. 
\end{theorem}

\begin{proof}
Since $J/K$ is not ramified at any finite prime, we have $e_p(J|K) = 1$,
then the ramification indices coincide,
thus $|X_p|=|Y_p|$ for all primes $p.$ Since $X_p \subseteq Y_p$, we have
$X_p=Y_p.$ Let $Z :=\prod_{p\in{\mathcal P}} X_p$. Then $Z_p=X_p$. Let $F$ be
the field associated to $Z$. As $X \subseteq \prod_{p\in{\mathcal P}} X_p = Z$,
we have $K \subseteq F$ and analogously $J \subseteq F$.
On the other hand, since $|X_p|=|Z_p|$, the extension $F/K$ is unramified, thus
$F \subseteq J$. Therefore $F = J$ and it follows that $Y = Z = \prod_{p\in{\mathcal P}}
 X_p$.
\end{proof}

\begin{remark}\label{R2.2} {\rm{If the infinite primes are
unramified in $J/K$ we have $K_{{\eu {ge}}}=J$. Otherwise, $K$ is real
and $J$ is imaginary. Then $K_{{\eu {ge}}}=J^+$ where $J^+
:=J\cap {\ma R}$ and the group of Dirichlet characters
associated to $J^+$ is $Y^+:=\{\chi\in Y\mid \chi(-1)=1\}$.
Finally $[J:J^+]=[Y:Y^+]=2$.
}}
\end{remark}

\begin{example}[Gauss genus theorem]\label{E2.3}{\rm{ Let
$K={\ma Q}(\sqrt{d})$ be a quadratic extension of ${\ma Q}$,
where $d\in{\ma Z}$ is square free. Let $m$ be the number
of different prime factors of $\delta_K$, the discriminant
 of $K$. If $p_1,\ldots, p_m$ are these factors, we choose
$p_1=2$ if $2\mid \delta_K$.

Let $\chi$ be the quadratic character associated to $K$. Then
$\chi_{p_i}\neq 1$, $1\leq i\leq m$ and $\chi_q=1$ for all
$q\in {\mathcal P}\setminus \{p_1,\ldots,p_m\}$. For $p_i\neq 2$,
$\chi_{p_i}$ is unique and $\chi_{p_i}(-1)=(-1)^{(p_i-1)/2}$.
In this case the field associated to $\chi_{p_i}$ is ${\ma Q}\big(
\sqrt{(-1)^{(p_i-1)/2}p_i}\big)$. If $p_1=2$, then there are
three quadratic characters $\chi_{p_1}=\chi_2$; two of
them have conductor $8$, one is real and one imaginary, and
the other one has conductor $4$. If $\chi_2$ is real, $\chi(-1)=1$ and the
field associated is ${\ma Q}(\sqrt{2})$. If $\chi_2$ is imaginary
of conductor $8$, $\chi(-1)=-1$ and the field associated is
${\ma Q}(\sqrt{-2})$. Finally, if $\chi_2$ is of conductor $4$, 
$\chi(-1)=-1$ and the field associated to $\chi_2$ is ${\ma Q}
(\zeta_4)={\ma Q}(i)={\ma Q}(\sqrt{-1})$. It follows
that the maximal abelian extension of ${\ma Q}$ unramified
at every finite prime is $J={\ma Q}\big(\sqrt{\varepsilon}, \sqrt{(-1)^{
(p_i-1)/2}p_i}\mid 2\leq i\leq m\big)$ where $\varepsilon=
(-1)^{(p_1-1)/2}p_1$ if $p_1\neq 2$ and $\varepsilon =2, -2$ or $-1$ if 
$p_1=2$.

Thus we obtain $[J:{\ma Q}]=2^m$ and $[J:K]=2^{m-1}$. We have
$K_{{\eu {ge}}}=J$ except when $K$ is real and $J$ is imaginary and this
last case occurs when $\delta_K>0$ ($d>0$) and there exists
$p_i\equiv 3\bmod 4$. In this case, $[J^+:K]=2^{m-2}$.
For the quadratic extension $K={\ma Q}(\sqrt{-14})$ over ${\ma Q}$,
we have $K_{{\eu {ge}}}={\ma Q}(\sqrt{2},\sqrt{-7})$ and for $K=
{\ma Q}(\sqrt{79})$ we obtain $J={\ma Q}(\sqrt{-79},i)$ and
$K_{{\eu {ge}}}=J^+=J\cap {\ma R}={\ma Q}(\sqrt{79})=K$.

Now if ${\mathcal C}_K$ is the class group of $K$, ${\mathcal C}_K
\cong \Gal(K_H/K)$ and $E$ is the fixed field of ${\mathcal C}_K^2$, then
$\Gal(E/K)\cong {\mathcal C}_K/{\mathcal C}_K^2$. Since $K_{{\eu {ge}}}$ is the maximal
abelian extension of ${\ma Q}$ contained in $K_H$, $K_{{\eu {ge}}}$ is the
fixed subfield of $K_H$ under the derived group $G'$ of $G
:=\Gal(K_H/{\ma Q})$.
It can be verified that $G'={\mathcal C}_K^2$ so that $K_{{\eu {ge}}}=E$ and it
follows that the $2$--rank of ${\mathcal C}_K$ is ${m-1}$ unless $d>0$
and there exists a prime $p\equiv 3\bmod 4$ dividing $d$ and in
this case the $2$--rank of ${\mathcal C}_K$ is ${m-2}$.
}}
\end{example}

\begin{example}\label{E2.4}{\rm{ If $p$ is an odd prime, $K$ is a cyclic
extension of ${\ma Q}$ of degree $p$ and $m$ is the number
of ramified primes in $K$, it follows that $K_{{\eu {ge}}}$ is an elementary
abelian $p$--extension of ${\ma Q}$ of degree $p^m$ and $[K_{{\eu {ge}}}:K]=
p^{m-1}$. In particular $p^{m-1}\mid |{\mathcal C}_K|$.
}}
\end{example}

Now let $K$ be any abelian extension of ${\ma Q}$ with Dirichlet
character group $X$. Consider for each $p\in{\mathcal P}$, 
$X_p$. Let $J$ be the field associated to $\prod_{
p\in{\mathcal P}} X_p$. 
Let $p^{m_p}:=\gcd\{{\eu f}_{\chi_p}\mid \chi\in X\}$ where
${\eu f}_{\chi_p}$ denotes the conductor of $\chi_p$. Then
the field $K_p$ associated to $X_p$ is contained in $\cic {p^{m_p}}$
but not in $\cic {p^{m_p-1}}$. If $p$ is odd, $K_p$ is the unique
subfield of $\cic {p^{m_p}}$ of degree $|X_p|$ over ${\ma Q}$ and
$K_p/{\ma Q}$ is a cyclic extension. If $p=2$, $K_2$ is one of the
following fields. If $|X_2|=\varphi(2^{m_2})=2^{m_2-1}$, $K_2=
\cic {2^{m_2}}$. If $|X_2|=\frac{\varphi(2^{m_2})}{2}=2^{m_2-2}$,
$K_2=\cic {2^{m_2}}^+=
{\ma Q}\big(\zeta_{2^{m_2}}+\zeta_{2^{m_2}}^{-1}\big)=\cic {2^{m_2}}
\cap {\ma R}$ if $\chi(-1)=1$ for all $\chi\in X$ and $K_2=
{\ma Q}\big(\zeta_{2^{m_2}}-\zeta_{2^{m_2}}^{-1}\big)$ if there exists
$\chi\in X$ with $\chi(-1)=-1$. 

Therefore, if $K$ and $J$ are both real or both imaginary, 
$K_{{\eu {ge}}}=J=\prod_{p\in{\mathcal P}}K_p$. If $K$ is real and
$J$ is imaginary, $K_{{\eu {ge}}}=J^+=J\cap {\ma R}$.

\section{Cyclotomic function fields}\label{S3}

Most of the results on cyclotomic function fields we need
in this paper were developed by D.R. Hayes in \cite{Hay74}.
As a reference we use \cite{Hay74,Vil2006}. Let $k={\ma F}_q(T)$ be
a rational congruence function field, ${\ma F}_q$ denoting the finite
field of $q$ elements. Let $R_T={\ma F}_q[T]$ be the ring of 
polynomials, that is, we choose $R_T$ as
 the ring of integers of $k$. $R_T^+$ denotes the set of
monic irreducible polynomials in $R_T$. For $N\in R_T\setminus
\{0\}$, $\Lambda_N$ denotes the $N$--torsion of the Carlitz
module and $k(\Lambda_N)$ denotes the $N$--th cyclotomic
function field. The $R_T$--module $\Lambda_N$ is cyclic and $\lambda_N$,
or $\lambda$ if there is no possible confusion,
denotes a generator of $\Lambda_N$ as $R_T$--module.
If we let $X:=1/T$, $R_X=R_{1/T}={\ma F}_q[1/T]$, then
$k={\ma F}_q(1/T)$ and we define $\Lambda_{1/T^n}$ as
the $(1/T^n)$--torsion of the Carlitz module with $R_X$ instead of $R_T$.
For any function field $K/{\ma F}_q$, $K_m:= K{\ma F}_{q^m}$
denotes the constant field extension. For any $m\in{\ma N}$,
$C_m$ denotes a cyclic group of order $m$.

We have $k(\Lambda_N)=k(\lambda_N)$ and $G_N:=
\Gal(k(\Lambda_N)/k)\cong \G N$ with the identification
$\sigma_A\lambda_N=\lambda_N^A$ for $A\in R_T$.
For any finite extension $K/k$ we will use the symbol
$\S K$ to denote
either one prime or the set of all primes in $K$ above ${\eu p}_{\infty}$,
the pole divisor of $T$ in $k$.
When we mention the degree of $\S K$, where $K/k$ is a Galois
extension, we mean the degree of each element of $\S K$.
 We understand by a {\em
Dirichlet character} any group homomorphism $\chi\colon
\G N\to {\ma C}^{\ast}$ and we define the conductor ${\eu f}_{\chi}$
of $\chi$ as the monic polynomial of minimum degree such that
$\chi$ can be defined modulo ${\eu f}_{\chi}$, $\chi\colon
\G {{\eu f}_{\chi}}\to {\ma C}^{\ast}$.

Given any group of characters $X\subseteq \widehat{G_N}(=\Hom(G_N, 
{\ma C}^{\ast}))$, the field associated to $X$ is the subfield of $k(\Lambda_N)$
fixed under $\cap_{\chi\in X}\ker \chi$. Conversely, for any field $K
\subseteq k(\Lambda_N)$, the group of Dirichlet characters associated to 
$K$ is $\widehat{\Gal(K/k)}$.

For any character $\chi$ we consider the canonical decomposition
$\chi=\prod_{P\in R_T^+}\chi_P$, where $\chi_P$ has conductor a power of
$P$. We have ${\eu f}_{\chi}=\prod_{P\in R_T^+} {\eu f}_{\chi_P}$.

If $X$ is a group of Dirichlet characters, we write $X_P:=\{
\chi_P\mid \chi\in X\}$ for $P\in R_T^+$. If $K$ is any extension
of $k$, $k\subseteq K\subseteq k(\Lambda_N)$ and $P\in R_T^+$,
then the ramification index of $P$ in $K$ is $e_P=|X_P|$.

In $k(\Lambda_N)/k$, ${\eu p}_{\infty}$ has ramification index $q-1$
and decomposes into $\frac{|G_N|}{q-1}$ different
prime divisors of $k(\Lambda_N)$ of degree $1$. Furthermore,
with the identification $G_N\cong \G N$, the inertia
group ${\eu I}$ of ${\eu p}_{\infty}$ is ${\ma F}_q^{\ast}\subseteq \G N$,
that is, ${\eu I}=\{\sigma_a\mid a\in {\ma F}_q^{\ast}\}$. In this
case the inertia and the decomposition groups coincide. The
primes that ramify in $k(\Lambda_N)/k$ are ${\eu p}_{\infty}$
and the polynomials $P\in R_T^+$ such that $P\mid N$.

We set $L_n$ to be the largest
subfield of $k(\Lambda_{1/T^n})$ where ${\eu p}_{
\infty}$ is fully and purely wildly ramified, $n\in{\ma N}$. For any field $F$,
$_nF$ denotes the composite $FL_n$.

We recall Rosen's definition for a relative Hilbert class field of a
congruence function field $K$.

\begin{definition}[\cite{Ros87}]\label{D3.1}{\rm{
Let $K$ be a function field with field of constants ${\ma F}_q$.
Let $S$ be any nonempty finite set of prime divisors of $K$. The
{\em Hilbert class function field of $K$ relative to $S$}, $K_{H,S}$, is the
maximal unramified abelian extension of $K$ where every element
of $S$ decomposes fully.
}}
\end{definition}

From now on, for any finite extension $K$ of $k$ we will consider
$S$ as the set of prime divisors dividing ${\eu p}_{\infty}$, the pole
divisor of $T$ in $k$ and we write $K_H$ instead of $K_{H,S}$.

\begin{definition}\label{D3.2}{\rm{
Let $K$ be a finite geometric extension of $k$, that is,
the exact field of constants of $K$ is ${\ma F}_q$.
The {\em genus
field} $K_{{\eu {ge}}}$ of $K$ is the maximal extension of $K$
contained in $K_H$ that is the composite
of $K$ and an abelian extension of $k$. Equivalently,
$K_{{\eu {ge}}}=K k^{\ast}$ where $k^{\ast}$ is the maximal abelian extension
of $k$ contained in $K_H$.
}}
\end{definition}

When $K/k$ is an abelian extension, $K_{{\eu {ge}}}$ is the maximal
abelian extension of $k$ contained in $K_H$. Our main goal
in this section
is to find $K_{{\eu {ge}}}$ when $K$ is a subfield of a cyclotomic
function field.
In what follows
$K$ will always denote a finite geometric
abelian extension of 
$k$. First we note that we have the analogue to Leopoldt's
result.

\begin{proposition}\label{P3.3}
If $K\subseteq k(\Lambda_N)$ and the group of
characters associated to $K$ is $X$, then the maximal
abelian extension $J$ of $K$ unramified at every finite prime
$P\in R_T^+$, contained in a cyclotomic extension, is the
field associated to $Y=\prod_{P\in R_T^+} X_P=\prod_{
P\mid N} X_P$.
\end{proposition}

\begin{proof}
Analogous to the proof of Theorem \ref{T2.1}.
\end{proof}

In this case ${\eu p}_{\infty}$ has no inertia in $J/K$ but it might
be ramified.

The following proposition should be well known. However, since we
could not find any reference, we include it here.

\begin{proposition}\label{P3.4} If $E/k$ is an abelian extension such
that ${\eu p}_{\infty}$ is tamely ramified, then there exist $N\in R_T$
and $m\in{\ma N}$ such that $E\subseteq k(\Lambda_N){\ma F}_{q^m}$.
\end{proposition}

\begin{proof}
By the Kronecker--Weber Theorem \cite[Theorem 12.8.5]{Vil2006}, we have
$E\subseteq k(\Lambda_N){\ma F}_{q^m}L_n={_n\lam N}_m$ for some $N\in R_T$ and
$n,m\in{\ma N}$.
\[
\xymatrix{F\ar@{-}[rrr]\ar@{-}[dd]&&&FL_n\ar@{-}[dl]_V\ar@{-}[dd]\\
&E\ar@{-}[dl]&R\\k\ar@{-}[rrr]&&&L_n}
\]
Let $F:= k(\Lambda_N) {\ma F}_{q^m}={\lam N}_m$ and let $V$ be
 the first ramification group of
${\eu p}_{\infty}$ in $FL_n/k$. Then $R:=(FL_n)^V$ is the maximal extension
of $k$ contained in $FL_n$ where ${\eu p}_{\infty}$ is tamely ramified and in 
consequence $\S R$
is wildly ramified in $FL_n/R$.

Since ${\eu p}_{\infty}$ is tamely ramified in $E/k$, it follows that $E
\subseteq R$. Now, ${\eu p}_{\infty}$ is tamely ramified in $F/k$ and
$\S F$ is fully and wildly ramified in $FL_n/F$ and $FL_n/F$ is of degree $|V|$.
Hence $R=F$ and $E\subseteq F$.
\end{proof}

\begin{proposition}\label{P3.5} With the hypothesis of 
Proposition {\rm{\ref{P3.3}}}, if $e_{{\eu p}_{\infty}}(K|k)=q-1$, then
$K_{{\eu {ge}}}=J$.
\end{proposition}

\begin{proof}
Since $e_{{\eu p}_{\infty}}(J|K)=\frac{e_{{\eu p}_{\infty}}
(J|k)}{e_{{\eu p}_{\infty}}(K|k)}=\frac{q-1}{q-1}=1$,
${\eu p}_{\infty}$ decomposes fully in $J/K$ and therefore
$J\subseteq K_{{\eu {ge}}}$.

Now the field of constants of $K_{{\eu {ge}}}$ is ${\ma F}_q$ (see 
\cite{Ros87} or simply if ${\ma F}_{q^m}$ is the field of constants
of $K_{{\eu {ge}}}$, $k\subseteq k_m\subseteq K_{{\eu {ge}}}$ and $\p$
is fully inert in $k_m$; since $\p$ and 
the primes in $\S K$ have no inertia in
either $K/k$ or $J/K$, $m=1$.)

Since ${\eu p}_{\infty}$ decomposes fully in $K_{{\eu {ge}}}/K$ and 
${\eu p}_{\infty}$ is tamely ramified in $K/k$, by Proposition
\ref{P3.4} we have $K_{{\eu {ge}}}\subseteq k(\Lambda_N) {\ma F}_{q^m}$
for some $N\in R_T$ and $m\in{\ma N}$.

In all the extensions $k_m/k$, $K_m/K$, $J_m/J$,
$k(\Lambda_N)_m/k(\Lambda_N)$ the infinite primes are
fully inert since all have degree $1$ (see
\cite[Theorem 6.2.1]{Vil2006}).
In the extensions $K_m/k_m$ and $K/k$ the ramification index
of the infinite primes is $q-1$, that is, the maximal possible.
It follows that in $J_m/K_m$, $J/K$, $k(\Lambda_N)/J$ and
$k(\Lambda_N)_m/J_m$, $\S {K_m}$, $\S K$, $\S J$ and
$\S {J_m}$ are fully decomposed. Finally, in $k(\Lambda_N)_m
/J$ (and therefore in $k(\Lambda_N)_m/K_{{\eu {ge}}}$), $\S J$ is
unramified.

Let ${\mathcal G}:=\Gal(k(\Lambda_N)_m/J)$. For
$\S J$ we have that in this extension
 the ramification index $e$,
the inertia degree $f$ and the decomposition number $h$ are
$e=1$, $f=m$ and $h=\frac{|{\mathcal G}|}{m}$. Therefore
the decomposition group ${\eu D}$ of ${\eu p}_{\infty}$ is
of order $m$ and it is cyclic. We must have ${\eu D}=\Gal
(k(\Lambda_N)_m/k(\Lambda_N))$ because $\S {k(\Lambda_N)}$
is fully inert of degree $m$ in $k(\Lambda_N)_m/
k(\Lambda_N)$. Since the primes in $\S J$ have inertia
degree $1$ in $K_{{\eu {ge}}}/J$, it follows that $K_{{\eu {ge}}}\subseteq
k(\Lambda_N)$. Thus $K_{{\eu {ge}}}=J$.
\end{proof}

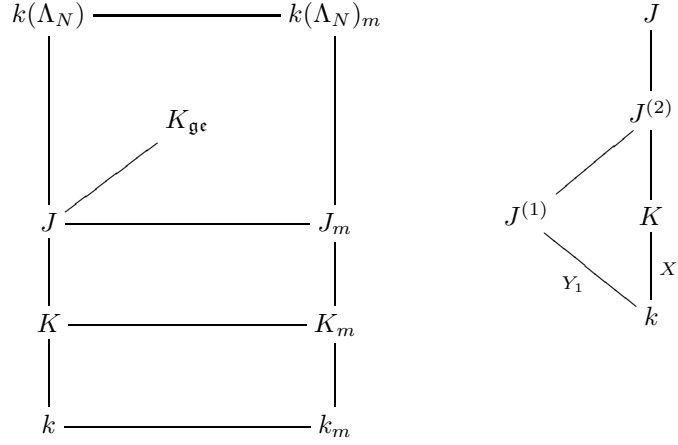
\[
\xymatrix{k(\Lambda_N)\ar@{-}[rr]\ar@{-}[dd]&&
k(\Lambda_N)_m\ar@{-}[dd]\\ &K_{{\eu {ge}}}\ar@{-}[dl]\\
J\ar@{-}[rr]\ar@{-}[d]&&J_m\ar@{-}[d]\\
K\ar@{-}[rr]\ar@{-}[d]&&K_m\ar@{-}[d]\\
k\ar@{-}[rr]&&k_m}
\qquad\qquad
\xymatrix{
&J\ar@{-}[d]\\&J^{(2)}\ar@{-}[d]\ar@{-}[dl]\\
J^{(1)}\ar@{-}[rd]_{Y_1}&K\ar@{-}[d]^X\\&k}
\]

Now we consider the case $k\subseteq K\subseteq
k(\Lambda_N)$ where $e_{\p}(K|k)$ not necessarily is equal to
$q-1$. We use the notations of Proposition \ref{P3.3}.
In this case $\S K$ might be ramified in $J/K$. Let
$Y_1:=\big\{\chi\in Y\mid \chi(a)=1 \text{\ for all\ } a\in {\ma F}_q^{\ast}
\subseteq \G N\cong G_N\big\}$ and let $J^{(1)}$ be the field
associated to $Y_1$. Then $J^{(1)}\subseteq J$ since $Y_1
\subseteq Y$, though not necessarily $J^{(1)}\subseteq K$ or
$K\subseteq J^{(1)}$. Let $J^{(2)}:=KJ^{(1)}$.

Then $J^{(2)}$ is the
field associated to the character group $X Y_1$. Since
${\eu p}_{\infty}$ decomposes fully in $J^{(1)}/k$, $\S K$
decomposes fully in $J^{(2)}$. Furthermore $\S {J^{(1)}}$
is fully ramified in $J/J^{(1)}$.
Hence $\S {J^{(2)}}$ is fully ramified in $J/J^{(2)}$.

We obtain that $J^{(2)}/K$ is an unramified abelian extension
with $J^{(2)}\subseteq k(\Lambda_N)$ and $\S K$ decomposes
fully in $J^{(2)}/K$. It follows that $J^{(2)}=J^{\eu D}$
where ${\eu D}$ is
the decomposition group of any prime in
$\S J$ with respect to the
Galois group $\Gal(J/K)$.

Now consider any unramified abelian extension $F/K$
such that $\S K$ decomposes fully in $F$.
By Proposition \ref{P3.4}, $F\subseteq k(\Lambda_N)
{\ma F}_{q^m}$ for some $N\in R_T$ and $m\in{\ma N}$. In
case $F\subseteq k(\Lambda_N)$, let $Z$ be the group
of Dirichlet characters associated to $F$. Since $F/K$
is unramified, it follows that $X\subseteq Z\subseteq Y$
by Proposition \ref{P3.3} and thus $F\subseteq J$. Since
$J^{(2)}=J^{\eu D}$, we obtain that $F\subseteq J^{(2)}$.
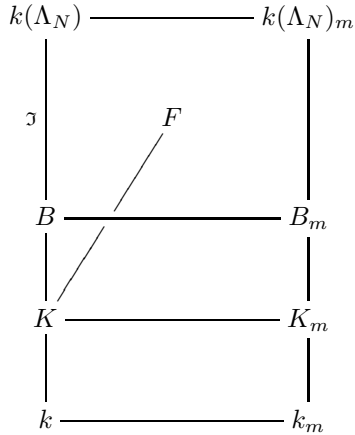
\[
\xymatrix{k(\Lambda_N)\ar@{-}[dd]_{\eu I}\ar@{-}[rr]&&k(\Lambda_N)_m
\ar@{-}[dd]\\ &F\ar@{-}[ddl] |!{[dl];[dr]}\hole\\  B\ar@{-}[d]\ar@{-}[rr]&&B_m\ar@{-}[d]\\
K\ar@{-}[rr]\ar@{-}[d]&&K_m\ar@{-}[d]\\
k\ar@{-}[rr]&&k_m}
\]

For the case $k\subseteq F\subseteq k(\Lambda_N)
{\ma F}_{q^m}$, that is, $F$ not necessarily is contained in
a cyclotomic function field, let ${\eu I}$ be the inertia group
of $\S K$ in $k(\Lambda_N)/K$ and let $B:=k(\Lambda_N)^{\eu I}$. Then
 the primes in $\S B$ are fully inert in $B_m$ because 
 they have degree $1$ and they are
 fully ramified in $\lam N/B$. Since $\S K$ decomposes fully
in $B$, $B$ is the decomposition field of the primes in $\S K$ in 
$k(\Lambda_N)_m/K$ so $F\subseteq B\subseteq k(\Lambda_N)$.

From the first part, we obtain that $F\subseteq J^{(2)}$. We have proved
the following

\begin{theorem}\label{T3.6} Assume $K\subseteq \lam N$ for some
polynomial $N$. Let $X$ be the group of Dirichlet characters associated
to $K$, $Y=\prod_{P\mid N}X_P$, $Y_1=\{\chi\in Y\mid \chi(a)=1
\text{\ for all\ } a\in{\ma F}_q^{\ast}\}$ and $J^{(1)}$ the field associated
to $Y_1$. Then the genus field $K_{{\eu {ge}}}$ of $K$ satisfies
$K_{{\eu {ge}}}\subseteq k(\Lambda_N)$ and $K_{{\eu {ge}}}=KJ^{(1)}$.
\hfill \qed
\end{theorem}

\section{General congruence function fields}\label{S4}

First we prove the following result.

\begin{lemma}\label{L4.1.1}
If $K/k$ is an abelian extension and
the degree of any prime divisor in $\S K$ is $t$, then the field
of constants of $K_{{\eu {ge}}}$ is ${\ma F}_{q^t}$.
\end{lemma}

\begin{proof}
Consider the constant field extension $K_r:=K{\ma F}_{q^r}$ of $K$. Then the
number of primes in $K_r$ above any prime in $\S K$ is $h=\gcd
(d_K(\S K),r)=\gcd(t,r)$ (\cite[Theorem 6.2.1(2)]{Vil2006}). Therefore $\S K$
decomposes fully in $K_r/K$ iff $h=r$ and this is equivalent to $r\mid
d_K(\S K)=t$. It follows that the maximal constant field extension of $K$
where $\S K$ decomposes fully is $K_t=K{\ma F}_{q^t}$. Thus the
field of constants of $K_{{\eu {ge}}}$ is ${\ma F}_{q^t}$.
\end{proof}

\subsection{Congruence function fields where $\p$
is tamely ramified}\label{S4.1}

Now we consider any finite geometric abelian extension  $K/{\ma F}_q$ of $k$
such that $\p$ is tamely ramified. Then we have $K\subseteq k(\Lambda_N)
{\ma F}_{q^m}=k(\Lambda_N)_m$ for some $N\in R_T$ and
$m\in{\ma N}$. Since $k(\Lambda_N)/k$ is a geometric
extension and $k_m/k$ is an extension of constants,
we have $k(\Lambda_N)\cap k_m=k$. Since $\p$ is tamely ramified
in $K_{{\eu {ge}}}/k$, without
loss of generality we may assume that $K_{{\eu {ge}}}\subseteq k(\Lambda_N)_m$.

Define $E:= K_m\cap k(\Lambda_N)\subseteq K_m$. 
By the Galois correspondence we have $E_m=E k_m=K_m$. We also
have $[E:k]=[K:k]$ since $m[K:k]=[K_m:k]=[E_m:k]=m[E:k]$. 
Hence,
\begin{gather}\label{Eq1Co}
E_m=K_m\quad\text{and}\quad [E:k]=[K:k].
\end{gather}
In other
words, $E$ plays a role similar to that of $K$ but it is
contained in a cyclotomic extension.

Since $E=K_m\cap \lam N$, it follows that $E\cap K=E_{{\eu {ge}}}\cap K=
\lam N\cap K$.
Because $K_m/K$ and $E_{{\eu {ge}}}/E$ are unramified, we obtain that
$E_{{\eu {ge}}}K/K$ is unramified. Also, since $\S E$ decomposes
fully in $E_{{\eu {ge}}}$, $\S{EK}$ decomposes fully in $E_{{\eu {ge}}} K$.
Now, $\S {E\cap K}$ has inertia degree one in $E/(E\cap K)$ but
might be ramified, so $\S K$ 
might have inertia in $EK/K$. Since the decomposition group of $\S K$
corresponding to the extension $EK/K$ is contained in the decomposition
group of $\S {E\cap K}$ corresponding to the extension $E/(E\cap K)$ (and
this group is equal to the inertia group), the inertia degree $d$ of $\S K$
in $EK/K$ divides the ramification index of $\S {E\cap K}$ in $E/(E\cap K)$.
This last one is a divisor of $q-1$. That is, $d\mid(q-1)$.

Let $H$ be the decomposition group of $\S K$ in $E_{{\eu {ge}}}K/K$. We have
that $H$ is a cyclic group of order $d$ and it corresponds to the inertia
of $\S K$ in $E_{{\eu {ge}}} K/K$. If $H_1:=H|_{E_{{\eu {ge}}}}$, we have
that $(\g E{} K)^H=\g E{}^{H_1}K$ is unramified over $K$ and $\S K$ splits
completely. Thus 
$(\g E{} K)^H\subseteq \g K{}$.
Note that since $\S E$ is fully decomposed in $\g E{}/E$ and
$\S {\g E{}^{H_1}}$ is fully ramified in $\g E{}/\g E{}^{H_1}$, 
it follows that $E \g E{}^{H_1}=\g E{}$. In short, we have
\begin{gather}\label{Eq2Co}
(\g E{} K)^H=\g E{}^{H_1}K\subseteq \g K{}\quad\text{and}\quad
E \g E{}^{H_1}=\g E{}.
\end{gather}

Finally, let $C:=K_{{\eu {ge}},m}\cap \lam N$. From (\ref{Eq1Co}) we have
$E\subseteq E_m=K_m\subseteq \g K{,m}$ and $E\subseteq\lam N$. Hence
$E\subseteq C$. From (\ref{Eq2Co}) we obtain $\g E{}^{H_1}\subseteq
\g E{}^{H_1} K\subseteq \g K{}\subseteq \g K{,m}$. Hence $\g E{}^{H_1}
\subseteq \g K{,m}\cap \lam N=C$. We also have $\g E{}=E \g E{}^{H_1}
\subseteq C$.

Now, by definition, $\g K{}/K$ is unramified. Thus $\g K{,m}/K_m$
is unramified. We have $K_m=E_m=
(EK)_m$ and $E_m/E$, being an extension of constants,
is unramified. It follows that $\g K{,m}/E$ is unramified. Since
$E\subseteq C\subseteq \g K{,m}$ we obtain that $C/E$ is unramified.
Finally, because $C\subseteq \lam N$ and $\S E$ is unramified in
$C/E$, it follows that $\S E$ is fully decomposed in $C/E$. Hence
$C\subseteq \g E{}$. Since 
$\lam N_m=\lam N k_m$ and $\lam N\cap k_m=k$,
by the Galois correspondence, we obtain 
\begin{gather}\label{Eq3Co}
C=\g E{}\quad\text{and}\quad \g E{,m}=C k_m=\g K{,m}.
\end{gather}
\begin{tiny}
\[
\xymatrix{
&k(\Lambda_N)\ar@{-}[rrrrr]\ar@{-}[d]&&&&&k(\Lambda_N)_m\ar@{-}[d]\\
&C\ar@{-}[rrrrr]\ar@{--}[d]^{C=E_{{\eu {ge}}}}
&&&&&K_{{\eu {ge}},m}\ar@{--}[d]_{K_{{\eu {ge}},m}=E_{{\eu {ge}},m}}\\
&E_{{\eu {ge}}}\ar@{-}[dl]_{H_1=H|_{\g E{}}}\ar@{-}[ddr]\ar@{-}[rrr]&&&
E_{{\eu {ge}}}K\ar@{-}[ddr]\ar@{-}[dl]_H\ar@{-}[rr]\ar@/_2pc/@{-}[ddd]
|!{[dl];[d]}\hole|!{[ddll];[ddr]}\hole
&&E_{{\eu {ge}},m}\ar@{-}[dd]\\
\g E{}^{H_1}\ar@{-}[ddr]\ar@{-}[rrr]|!{[ur];[drr]}\hole
&&&\g E{}^{H_1}K=(\g E{}K)^H\ar@{-}[ddr]\ar@{-}[r]&\g K{}
\ar@{-}[dd]|!{[dll];[dr]}\hole\\
&&E\ar@{-}[rrr]|!{[ru];[rrd]}\hole \ar@{-}[dl]&&&EK\ar@{-}[r]\ar@{-}[dl]
&E_{m}=K_m\ar@{-}[dd]\ar@{-}[dll]\\
&E\cap K\ar@{-}[rrr]\ar@{-}[d]&&&K\\
&k\ar@{-}[rrrrr]&&&&&k_m}
\]
\end{tiny}

From (\ref{Eq2Co}) we have $E_{{\eu {ge}}}^{H_1}
K\subseteq K_{{\eu {ge}}} \subseteq K_{{\eu {ge}},m}$ and
\[
(\g E{}^{H_1}K)_m=\g E{}^{H_1}K_m=\g E{}^{H_1}E_m=(\g E{}^{H_1} E)_m=
\g E{,m}=\g K{,m}.
\]

Thus  $K_{{\eu {ge}},m}/E_{{\eu {ge}}}^{H_1}K$ 
is an extension of constants.
From Lemma \ref{L4.1.1}, the field of constants of 
$K_{{\eu {ge}}}$ is ${\ma F}_{
q^t}$, so $K_{{\eu {ge}}}=(E_{{\eu {ge}}}^{H_1} K)_t$. 
If we prove that ${\ma F}_{q^t}
\subseteq E_{{\eu {ge}}}^{H_1}
 K$, it will follow that $K_{{\eu {ge}}}=E_{{\eu {ge}}}^{H_1} K$.

Let ${\eu I}$ be the inertia group of any element of $\S K$ in the
extension $K/k$, $|{\eu I}|=e(\S K\mid \p)=e$, and let ${\eu D}'$ be
the decomposition group of $\S K$ in $K/k$. We have $|{\eu D}'|=et$ since
${\eu D}'/{\eu I}\cong \Gal(K(\S K)/k(\p))\cong C_t$. In the following
diagram the decomposition type is referred to the infinite prime divisors.

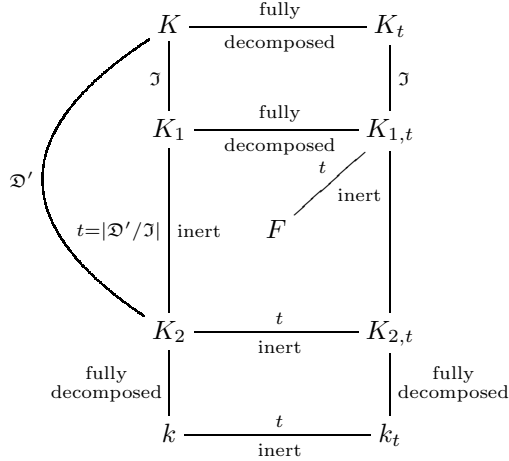
\[
\xymatrix{
K\ar@{-}[rr]^{\text{fully}}_{\text{decomposed}}
\ar@{-}[d]_{{\eu I}}\ar@/_4pc/@{-}[ddd]_{{\eu D}'}
 &&K_t \ar@{-}[d]^{{\eu I}}\\
K_1\ar@{-}[rr]^{\text{fully}}_{\text{decomposed}}\ar@{-}[dd]_{t=|{\eu D}'/{\eu I}|}^{
\text{inert}} && K_{1,t}\ar@{-}[dd]\ar@{-}[dl]_{t}^{\text{inert}}\\
&F\\
K_2\ar@{-}[rr]^t_{\text{inert}}\ar@{-}[d]_{\substack{\text{fully}\\ \text{decomposed}}}
&&K_{2,t}\ar@{-}[d]^{\substack{\text{fully}\\ \text{decomposed}}}\\
k\ar@{-}[rr]^t_{\text{inert}}&&k_t}
\]

Here $K_1:=K^{{\eu I}}$, $K_2:=K^{{\eu D}'}$ and $F$ is the fixed field
of the decomposition group of $\p$ in $K_{1,t}/k$. It follows that $\p$
is fully decomposed in $F/k$. Therefore $F\subseteq \lam N$.
The inertia degree of any element of $\S {K_2}$  in
the extension $K_1/K_2$ is $t$. Thus $F\cap K_1=K_2$ and $FK_1/K_2$
is an extension of degree $t^2$ with Galois group $C_t\times C_t$.
In particular we obtain $FK_1=K_{1,t}$. 

Since $F\subseteq K_t$ and $K\subseteq \lam N_m$, we have
$F\subseteq K_t\subseteq (\lam N_m)_t=\lam N_m$. Because
$E=K_m\cap \lam N$ it follows that $F\subseteq E$
and since $\p$ is fully decomposed in $F/k$, we have 
$F\subseteq E^{H_1}$. Therefore 
$k_t K_1=K_{1,t}=FK_1\subseteq FK\subseteq E^{H_1}K
\subseteq \g E{}^{H_1}K$. Thus
the field of constants of $E^{H_1}K$ contains ${\ma F}_{q^t}$. Hence
\begin{gather}\label{Eq4Co}
K_{{\eu {ge}}}=E_{{\eu {ge}}}^{H_1} K.
\end{gather}

Now, from (\ref{Eq2Co}) we obtain
$\g E{}^{H_1}K=\g K{}\subseteq \g E{}K\subseteq (\g E{}K)_m=
\g E{}^{H_1} EK_m=(\g E{}^{H_1}K)_m$. In particular, $\g E{}K/\g K{}$
is an extension of constants. Since $K\cap \g E{}=K\cap E$, from the
Galois correspondence we obtain 
\[
[\g E{}K:\g K{}]=[\g E{}K:\g E{}^{H_1}K]=[\g E{}:\g E{}^{H_1}]=d=|H_1|=|H|
\]
and since the field of constants of $\g K{}$ is ${\ma F}_{q^t}$, it follows
that the field of constants of $\g E{}K$ is ${\ma F}_{q^{td}}$ and
\begin{gather}\label{Eq4'Co}
\g E{}K=\g K{} {\ma F}_{q^{td}}.
\end{gather}

Finally, from (\ref{Eq3Co}) we have 
$K_{{\eu {ge}}}=E_{{\eu {ge}},m}^{\eu D}=\g K{,m}^{\eu D}$ where
${\eu D}$ is the decomposition group of the prime divisors in
 $\S K$ in $E_{{\eu {ge}},m}/K$. Observe that $|{\eu D}|=
[\g E{,m}:\g K{}]= [K_{{\eu {ge}},m}:K_{{\eu {ge}}}]
=\frac{m}{t}$ where $t$ is the degree of any prime in $\S K$.

We have proved

\begin{theorem}\label{T4.1} Let $K/{\ma F}_q$ be a geometric finite abelian
extension of $k$ where $\p$ is tamely ramified. Let $N\in R_T$
and $m\in{\ma N}$ be such that $\g K{}\subseteq {\lam N}{\ma F}_{q^m}$.
Let $E_{{\eu {ge}}}$ be the genus field of $E:=k(\Lambda_N)\cap
K{\ma F}_{q^m}$ and let $E_{{\eu {ge}},m}=E_{{\eu {ge}}}{\ma F}_{q^m}$. 
Let $H$ be the decomposition group of $\S K$ in $\g E{}K/K$. Let $H_1
=H|_{\g E{}}$. Let
${\eu D}$ be the decomposition group of the prime divisors
in $\S K$ in $E_{{\eu {ge}},m}$. Then
the genus field of $K$ is 
$$
K_{{\eu {ge}}}= (\g E{}K)^H=\g E{}^{H_1}K=\g E{,m}^{{\eu D}}.
\eqno{\qed}
$$
\end{theorem}

\begin{remark}\label{R4.2}{\rm{
Let $\langle\sigma\rangle =
\Gal(k(\Lambda_N)_m/k(\Lambda_N))\cong
\Gal(k_m/k)$. Then with the above notations, we have
${\eu D}\cong \langle \sigma^t\rangle$ and
\begin{gather}\label{Eq5Co}
[K_{{\eu {ge}}}:K]=\frac{[\g K{,m}:K]}{[\g K{,m}:\g K{}]}=
\frac{[E_{{\eu {ge}},m}:K_m][K_m:K]}
{|{\eu D}|}=\frac{[E_{{\eu {ge}},m}:E_m]m}{m/t}=
[E_{{\eu {ge}}}:E]t,
\end{gather}
where $t$ is the degree of any prime in $\S K$.
}}
\end{remark}

\subsection{Congruence function fields with general
type ramification of $\p$}\label{S4.2}

Finally we consider any geometric finite abelian extension $K$ of $k$. By the
Kronecker--Weber Theorem, we have
$K\subseteq \lam N {\ma F}_{q^m} L_n={_n\lam N}_m$ for some $N
\in R_T$ and $n,m\in{\ma N}$. Let ${\mathcal G}
:=\Gal({_n\lam N}_m/{\lam N}_m)$,
${\mathcal H}:=\Gal({_n\lam N}_m/K{\lam N}_m)$, $M:=K{\lam N}_m\cap L_n=
L_n^{{\mathcal H}_1}$ where ${\mathcal H}_1:=
{\mathcal H}|_{L_n}$.

Let $G:=\Gal({_n\lam N}_m/
L_n)$, $H:=\Gal({_n\lam N}_m/{_nK})$, $F:= {_nK}\cap {\lam N}_m=
{\lam N}_m^{H_1}$ where $H_1:=H|_{{\lam N}_m}$.

We have $F={_nK}\cap {\lam N}_m\subseteq {_nK}$. Hence, on the one
hand $_nF\subseteq{_nK}$, and on the other hand $[{_n{\lam N}}_m:{_nF}]=
[{\lam N}_m:F]=|H_1|=|H|=[{_n{\lam N}}_m:{_nK}]$. It follows
that $_nF={_nK}$. Similarly we obtain $M{\lam N}_m=K{\lam N}_m$.

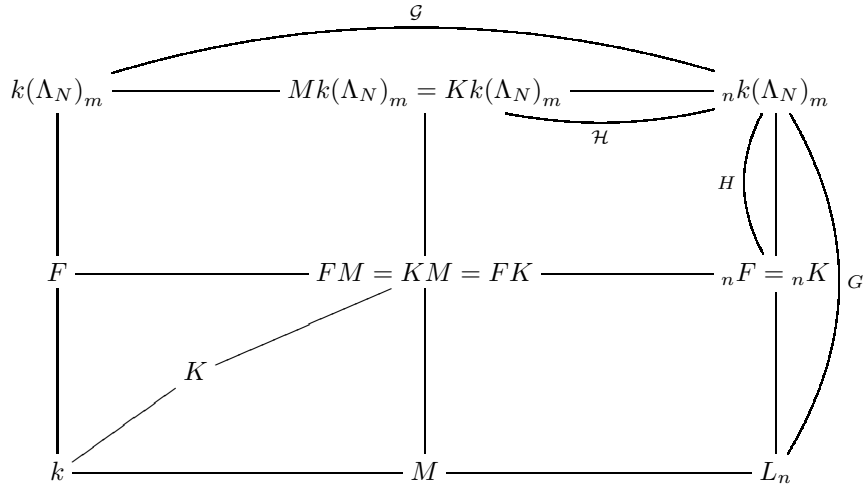
\[
\xymatrix{
{\lam N}_m\ar@{-}[rr]\ar@{-}[dd]\ar@/^2pc/@{-}[rrrr]^{{\mathcal G}}&&
M{\lam N}_m=K{\lam N}_m\ar@{-}[rr]\ar@/_1pc/@{-}[rr]_{{\mathcal H}}
\ar@{-}[dd]&&{_n{\lam N}_m}\ar@{-}[dd]\ar@/_1pc/@{-}[dd]_H
\ar@/^2pc/@{-}[dddd]^G\\ \\
F\ar@{-}[rr]\ar@{-}[dd]&&FM=KM=FK\ar@{-}[rr]\ar@{-}[dd]&&
{_nF}={_nK}\ar@{-}[dd]\\ &K\ar@{-}[dl]\ar@{-}[ru]\\
k\ar@{-}[rr]&&M\ar@{-}[rr]&&L_n
}
\]

Set $A\subseteq {\mathcal G}\times G$ such that $K={_n{\lam M}}_m^A$.
First we will prove that $FM=KM=FK$. We have $F={_n{\lam N}}_m^{
{\mathcal G}\times H}$ and $M={_n{\lam N}}_m^{
{\mathcal H}\times G}$. Then if we denote $R={_n{\lam N}_m}$, we have
\begin{gather*}
R^{A\cap ({\mathcal G}\times 1)}=R^A R^{{\mathcal G}\times 1}=
K{\lam N}_m=M{\lam N}_m\\
= R^{{\mathcal H}\times G} R^{
{\mathcal G}\times 1}=R^{({\mathcal H}\times{G})\cap
({\mathcal G}\times 1)}=R^{{\mathcal H}\times 1},
\end{gather*}
so that $A\cap ({\mathcal G}\times 1)={\mathcal H}\times 1$.
Similarly $A\cap(1\times G)=1\times H$.
Therefore
\begin{align*}
FM&= R^{{\mathcal G}\times H} R^{{\mathcal H}\times G}
=R^{({\mathcal G}\times H)\cap({\mathcal H}\times G)}
=R^{{\mathcal H}\times H},\\
KM&= R^AR^{{\mathcal H}\times G}=R^{A\cap ({\mathcal H}\times G)},\\
FK&=R^{{\mathcal G}\times H}R^A=R^{({\mathcal G}\times H)\cap A}.
\end{align*}
Since it is easily seen that $({\mathcal G}\times H)\cap A=
A\cap ({\mathcal H}\times G)={\mathcal H}\times H$, it
follows that $FM=KM=FK$.

Given that $F_{{\eu {ge}}}/F$ is unramified and $\S F$ decomposes fully, we obtain
that $_nK F_{{\eu {ge}}}/{_nK}$ is unramified and $\S {_nF}$ decomposes fully.
Now, in $_nK/K$ the only possible ramified primes are
those in $\S K$ and
if this is so, they are wildly ramified. It follows that in ${_nK}F_{{\eu {ge}}}/K$
the only possible ramified primes are the elements of
 $\S K$ and if this is so, they are
wildly ramified. In particular, in $F_{{\eu {ge}}}K/K$ the only possible
ramified primes are those in $\S K$ and if they are
ramified, they are wildly ramified.

Again, given that the extension $F_{{\eu {ge}}}/F$ is unramified and $\S F$ 
decomposes fully,  $F_{{\eu {ge}}} K/FK$ is unramified and $\S {FK}$
decomposes fully. In the extension $F/(K\cap F)$,
$\S{K\cap F}$ is tamely ramified, hence $\S K$ is
tamely ramified in $FK/K$. Therefore $\S K$ decomposes
fully in $FK/K$. In short, we have $F_{{\eu {ge}}}K\subseteq K_{{\eu {ge}}}$.

Since $FM=FK$, $F_{{\eu {ge}}}M=F_{{\eu {ge}}}K\subseteq K_{{\eu {ge}}}$.
Let $V$ be the first ramification group of $\p$ in $K_{{\eu {ge}}}/k$.
Set $E:=K_{{\eu {ge}}}^V$. Then $\p$ is tamely ramified in $E/k$
and therefore $E\subseteq {\lam N}_m$. 
Since $\p$ is fully
wildly ramified in $M/k$, $M\cap E=k$.
Now, being $V$ the first ramification group of $\p$ in the 
extension $K_{{\eu {ge}}}/k$ and it is so in $M/k$, it follows that
$\p$ and $\S M$ are tamely ramified in the
extensions $E/k$ and $K_{{\eu {ge}}}/M$ respectively
and $\S F$ is fully wildly ramified 
in the extension $FK=FM/F$.
In particular, 
$\S M$ is tamely ramified in $K_{{\eu {ge}}}/M$ and since $K_{{\eu {ge}}}/K$
is unramified, it follows that $K_{{\eu {ge}}}/FM$ is unramified.
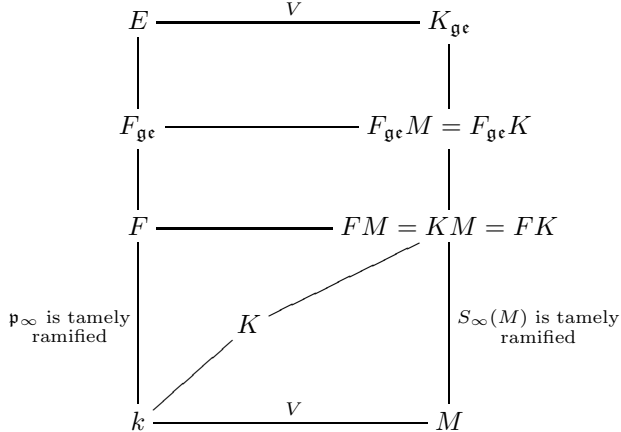
\[
\xymatrix{
E\ar@{-}[rr]^V\ar@{-}[d]&&K_{{\eu {ge}}}\ar@{-}[d]\\
F_{{\eu {ge}}}\ar@{-}[rr]\ar@{-}[d]&&F_{{\eu {ge}}}M=F_{{\eu {ge}}}K\ar@{-}[d]\\
F\ar@{-}[rr]\ar@{-}[dd]_{\substack{\p \text{\ is tamely}\\
\text{ramified}}}&&FM=KM=FK\ar@{-}[dd]^{
\substack{\S M\text{\ is tamely}\\ \text{ramified}}}\\
&K\ar@{-}[ru]\ar@{-}[dl]\\ k\ar@{-}[rr]^V&&M
}
\]

Now $[K_{{\eu {ge}}}:k]=[E:k]|V|=[E:k][M:k]=[EM:k]$. It
follows that $K_{{\eu {ge}}}=EM$. We also have $F_{{\eu {ge}}}\subseteq E$
because $F_{{\eu {ge}}}=F_{{\eu {ge}}}K\cap E\subseteq E$.
The extension $K_{{\eu {ge}}}/FK$ is unramified since $K_{{\eu {ge}}}/K$ is
unramified. Since $\p$ is the only ramified prime in $M/k$,
the only ramified primes in $FK=FM/F$ are those in
$\S F$ and as we just mentioned, they are wildly ramified.
$\S F$ is not ramified
in $E/F$ since otherwise it would be tamely ramified, and
$E/F$ is unramified at every other prime because $K_{{\eu {ge}}}/F$ is
ramified at most at the prime divisors in
$\S F$. It follows that $E\subseteq F_{{\eu {ge}}}$
and therefore $E=F_{{\eu {ge}}}$. Thus $K_{{\eu {ge}}}=EM=F_{{\eu {ge}}}M=F_{{\eu {ge}}} K$.

We have proved

\begin{theorem}\label{T4.3}
Let $K/k$ be any finite abelian extension with $K\subseteq
{_n\lam N}_m$. Let $F={_nK}\cap {\lam N}_m$ and
$M=K{\lam N}_m\cap L_n$. Then the genus field of $K$ is
$K_{{\eu {ge}}}= F_{{\eu {ge}}} K=F_{{\eu {ge}}} M$. \hfill\qed
\end{theorem}

Our main result is the combination of Theorems \ref{T4.1} and \ref{T4.3}.

\begin{theorem}\label{T4.4}
Let $K/k$ be a finite abelian extension with $K\subseteq \lam N
{\ma F}_{q^m} L_n$. Let $F=KL_n\cap {\lam N} {\ma F}_{q^m}$
and $E=\lam N\cap F{\ma F}_{q^m}\subseteq
k(\Lambda_N)$. Then the genus field of $K$
is 
\begin{gather*}
K_{{\eu {ge}}}=E_{{\eu {ge}}}^{H_1} FK
\end{gather*}
where $E_{{\eu {ge}}}$ is the genus field of $E$, 
$H$ is the decomposition group of $\S F$ in $\g E{}F/F$
and $H_1=H|_{\g E{}}$. Furthermore
$\g E{}FK/\g K{}$ is an extension of constants of degree
$d=|H|$ and $d$ divides $q-1$. Finally
$$
\g E{}FK=\g K{} {\ma F}_{q^{td}}
$$
where $t$ is the degree of $\S K$. \hfill $\qed$
\end{theorem}

\section{Applications}\label{S5}

In this section we will see how our results can be applied to some
general abelian extensions: Kummer, Artin--Schreier and $p$--cyclic
(Witt) extensions.

\subsection{Kummer Extensions}\label{S5.1}

Here we will assume that $q\geq 3$.
Let $P\in R_T^+$. Then $k(\sqrt[q-1]{(-1)^{\deg P}P})\subseteq
k(\Lambda_P)$ (see \cite[Exercise 5, page 303]{Ros2002}).
Thus for $l$ a prime
number such that $l\mid q-1$, we have $k(\sqrt[l]{(-1)^{\deg P}P})
\subseteq k(\Lambda_P)$.
Therefore for any monic polynomial $D\in R_T$,
we obtain $k(\sqrt[l]{(-1)^{\deg D} D})\subseteq k(\Lambda_D)$.

Note that for $\alpha, \beta\in{\ma F}_q^{\ast}$, we have
$k(\sqrt[l]{\alpha D})=k(\sqrt[l]{\beta D})$ iff $\alpha\equiv \beta
\mod ({\ma F}_q^{\ast})^l$. In particular $k(\sqrt[l]{\gamma D})
\subseteq k(\Lambda_D)$ iff $\gamma\equiv (-1)^{\deg D}
\bmod ({\ma F}_q^{\ast})^l$. It follows that if $l\mid \deg D$ then
$k(\sqrt[l]{D})\subseteq k(\Lambda_D)$.

In this subsection we use the notations
of Section \ref{S4}.
Let $K:=k(\sqrt[l]{\gamma D})$ with $D\in
R_T$ a monic $l$--power free polynomial, $\gamma\in \f$ and
$D=P_1^{e_1}\cdots P_r^{e_r}$ where $P_i\in R_T^+$, $1\leq e_i
\leq l-1$, $1\leq i\leq r$. Furthermore we arrange the product
so that $l\mid \deg P_i$ for $1\leq i\leq s$ and $l\nmid \deg
P_j$ for $s+1\leq j\leq r$, $0\leq s\leq r$. 
In general, we always have $E=k(\sqrt[l]{(-1)^{\deg D}D})$, and
$\f \subseteq ({\ma F}_{q^l}^{\ast})^l$. Now 
\[
K=E \iff (-1)^{\deg
D}\equiv \gamma \bmod ({\ma F}_q^{\ast})^l \iff (-1)^{\deg D}\gamma
\in ({\ma F}_q^{\ast})^l.
\]

\begin{proposition}\label{P5.1.2}
The behavior of $\p$ in $K/k$ is the following:
\l
\item If $l\nmid \deg D$, $\p$ is ramified.
\item If $l\mid \deg D$ and $\gamma\in (\f)^l$, $\p$ decomposes.
\item If $l\mid \deg D$ and $\gamma\not\in (\f)^l$, $\p$ is inert.
\end{list}
\end{proposition}

\begin{proof} 
\cite[Lemma 3]{Pen2003}.
\end{proof}

Note that in general $EK=k\big(\sqrt[l]{(-1)^{\deg D}D},\sqrt[l]{\gamma D}\big)
=K\big(\sqrt[l]{(-1)^{\deg D}\gamma}\big)=K{\ma F}_q\big(\sqrt[l]{(-1)^{\deg D}
\gamma}\big)$. We have ${\ma F}_q\big(\sqrt[l]{(-1)^{\deg D}\gamma}\big)=
{\ma F}_q\iff (-1)^{\deg D}\gamma\in ({\ma F}_q^{\ast})^l (\iff E=K)$. Consequently
${\ma F}_q\big(\sqrt[l]{(-1)^{\deg D}\gamma}\big)=
{\ma F}_{q^l}\iff (-1)^{\deg D}\gamma\notin ({\ma F}_q^{\ast})^l$. In any case
we have $EK=E$ or $EK=E_l=K_l$.

Now by (\ref{Eq5Co}) we have $[K_{{\eu {ge}}}:K]=[E_{{\eu {ge}}}:E]t$ where
\[
t=\deg\S K=
\begin{cases}
1&\text{if $\p$ is not inert in $K/k$},\\
l&\text{if $\p$ is inert in $K/k.$}
\end{cases}
\]

When $K=E$, that is, when $K\subseteq \lam D$, if $\chi$ is the
character of order $l$ associated to $K$, $\chi=\chi_{P_1}\cdots
\chi_{P_r}$, we consider $Y=\langle \chi_{P_i}\mid 1\leq i\leq r\rangle$.
The field associated to $Y$ is 
\[
F=k\big(\raizm 1,\ldots, \raizm r\big), 
\]
and $K_{{\eu {ge}}}=F$
if $l\nmid \deg D$ or if $l\mid \deg P_i$ for all $i$ (that is, $s=r$). This is 
because in the first case $\p$ is already ramified in $K$ and in the
second $\p$ is unramified in $F/k$ (Proposition \ref{P5.1.2}).

When $l\mid \deg D$ and $l\nmid \deg P_r$, $\p$ ramifies in $F/k$ and is
unramified in $E/k$. In this case $[F:E_{{\eu {ge}}}]=l$. Let $a_{s+1},\ldots,
a_{r-1}\in{\ma Z}$ be such that $l\mid \deg(P_mP_r^{a_m})$, that
is, $\deg P_m+a_m\deg P_r\equiv 0\bmod l$, $s+1\leq m\leq r-1$. Let
\[
F_1:=k\big(\raiz 1,\ldots, \raiz s, \sqrt[l]{P_{s+1}P_r^{a_{s+1}}},\ldots,
\sqrt[l]{P_{r-1}P_r^{a_{r-1}}}\big)\subseteq k(\Lambda_{P_1P_2\cdots P_r}).
\]
 Then $\S E$ decomposes in
$F_1/E$, $K\subseteq F_1\subseteq E_{{\eu {ge}}}$ and since
$F=F_1\big(\raizm r\big)$, we have $[F:F_1]=l$.
It follows that $E_{{\eu {ge}}}=F_1$.
We have
\begin{gather}\label{Eq5Co}
\g E{}=
\begin{cases}
F&\text{if $l\nmid \deg D$ or if $l \mid \deg P_r$,}\\
F_1&\text{if $l\mid \deg D$ and $l\nmid \deg P_r$.}
\end{cases}
\end{gather}

Let $\alpha:=(-1)^{\deg D}\gamma$. We have $E=K \iff \alpha\in\ne$
and $E\neq K\iff \alpha\notin \ne$. In particular, when $\alpha
\in\ne$ we have $\g E{}=\g K{}$. 

When $\alpha\notin\ne$, we
have
\[
EK=k(\sqrt[l]{(-1)^{\deg D} D},\sqrt[l]{\gamma D})=K(\sqrt[l]{\alpha})
=K_l\neq K.
\]

From Theorem \ref{T4.1} we obtain that if $\S K$ is not inert
in $EK/K$ then $\g K{}=\g E{}K$. When $\S K$ is inert 
in $EK/K$ then
$\g K{}=\g E{}^{H_1} K$ where $H_1$ is the inertia group
of $\p$ in $\g E{}/k$ .
This last case holds when $E\neq K$, that is, $\alpha\notin\ne$,
and, $\p$ is ramified in $E/k$, that is, $l\nmid \deg D$. Since $\p$
is unramified in $F_1/k$, it follows that $\g E{}^{H_1}=F_1$.
Combining this situation with (\ref{Eq5Co}), we obtain that
$\g K{}$ is equal to one of the following six cases:
\[
\begin{array}{lll}
F_1&\text{if $\alpha\in\ne$, $l\mid\deg D$ and $l\nmid \deg P_r$}&
\text{(this implies $\gamma\in\ne$)},\\
F&\text{if $\alpha\in\ne$ and $l\mid\deg P_r$}&\text{(this implies
$l\mid\deg D$ and $\gamma\in\ne$)},\\
F&\text{if $\alpha\in\ne$ and $l\nmid \deg D$},\\
F_1 K&\text{if $\alpha\notin\ne$, $l\mid\deg D$ and $l\nmid \deg P_r$}&
\text{(this implies $\gamma\notin \ne$)},\\
FK&\text{if $\alpha\notin\ne$ and $l\mid\deg P_r$}&\text{(this implies
$l\mid\deg D$ and $\gamma\notin\ne$)},\\
F_1K&\text{if $\alpha\notin\ne$ and $l\nmid \deg D$}&\text{(this
implies $l\nmid \deg P_r$)}.
\end{array}
\]

Therefore we have obtained

\begin{theorem}[G. Peng \cite{Pen2003}]\label{T5.1.6}
Let $D=P_1^{e_1}\cdots P_r^{e_r}\in R_T$ be a monic $l$--power free polynomial,
where $P_i\in R_T^+$, $1\leq e_i\leq l-1$, $1\leq i\leq r$. Let $0\leq s\leq r$
be such that $l\mid \deg P_i$ for $1\leq i\leq s$ and $l\nmid \deg P_j$
for $s+1\leq j\leq r$. Let $K:=k(\sqrt[l]{\gamma D})$ where $\gamma\in
\f$. Let $\alpha:=(-1)^{\deg D}\gamma$ and $a_{s+1},\ldots,a_{r-1}\in{\ma Z}$
satisfying $\deg P_m+a_m\deg P_r\equiv 0\bmod l$, $s+1\leq m\leq r-1$.
Then $K_{{\eu {ge}}}$ is given by:
\l
\item
$k\big(\raizm 1,\ldots, \raizm r\big)$ if $\alpha\in\ne$ and $l\nmid \deg D$.

\item
$k\big(\raiz 1,\ldots, \raiz s, \sqrt[l]{P_{s+1}P_r^{a_{s+1}}},\ldots,
\sqrt[l]{P_{r-1}P_r^{a_{r-1}}}\big)$ if $\alpha\in \ne$, $l\mid\deg D$, and
$l\nmid \deg P_r$.

\item
$k\big(\sqrt[l]{\gamma},\raiz 1,\ldots, \raiz r)$ if $l\mid\deg P_r$.

\item 
$k\big(\sqrt[l]{\gamma D}, 
\raiz 1,\ldots, \raiz s, \sqrt[l]{P_{s+1}P_r^{a_{s+1}}},\ldots,
\sqrt[l]{P_{r-1}P_r^{a_{r-1}}}\big)$ if $\alpha\notin\ne$
and $l\nmid \deg P_r$. \hfill\qed
\end{list}
\end{theorem}

\subsection{Artin--Schreier extensions}\label{S5.2}

Consider $K:=k(y)$ where $y^p-y=\alpha\in k$. 
The equation can be 
normalized as:
\begin{equation}\label{Eq1}
y^p-y=\alpha=\sum_{i=1}^r\frac{Q_i}{P_i^{e_i}} + f(T),
\end{equation}
where $P_i\in R_T^+$, $Q_i\in R_T$, 
$\gcd(P_i,Q_i)=1$, $e_i>0$, $p\nmid e_i$, $\deg Q_i<
\deg P_i^{e_i}$, $1\leq i\leq r$, $f(T)\in R_T$,
with $p\nmid \deg f$ when $f(T)\not\in {\ma F}_q$.

We have that the finite primes ramified in $K/k$ are precisely
$P_1,\ldots,P_r$. With respect to $\p$ we have

\begin{proposition}\label{P5.3}
The prime $\p$ is
\l
\item decomposed if $f(T)=0$,
\item inert if $f(T)\in {\ma F}_q$ and $f(T)\not\in \wp({\ma F}_q):=
\{a^p-a\mid a\in {\ma F}_q\}$,
\item ramified if $f(T)\not\in {\ma F}_q$ (thus $p\nmid\deg f$).
\end{list}
\end{proposition}

We study two cases.

\medskip

\noindent
{\bf Case 1}: We assume that $\p$ is not ramified, so $f(T)\in{\ma F}_q$.
Recall that $K_p=K{\ma F}_{q^p}$.
We have $\Gal(K_p/k)\cong C_p\times C_p$. 
The $p+1$ fields of degree $p$
over $k$ contained in $K_p$ are: $k(y+\beta_i)$, $1\leq i\leq p$ and $k_p$,
where $\big\{\beta_i\big\}_{i=1}^p$ is a basis of ${\ma F}_{q^p}$ over
${\ma F}_q$. By Proposition \ref{P5.3}, the unique such extension such
that $\p$ is not inert is the one $k(w)$ such that $w^p-w=\alpha-f(T)$. Thus
$E=k(w)$.

If $\chi$ is the character associated to $E$, then $\chi=\chi_{P_1}\cdots
\chi_{P_r}$ and the field associated to $\chi_{P_i}$ is $k(y_i)$,
where $y_i^p-y_i=\alpha_i:=\frac{Q_i}{P_i^{e_i}}$, $1\leq i\leq r$.
Therefore 
\begin{gather}\label{Eq3}
E_{{\eu {ge}}}=k(y_1,\cdots,y_r).
\end{gather}

Thus $K_{{\eu {ge}}}=E_{{\eu {ge}}}K=k(y_1,\ldots,y_r,\beta)$ with 
$\beta=0$ or $\beta^p-\beta
\not\in \wp({\ma F}_q)=\{x^p-x\mid x\in{\ma F}_q\}$.

\medskip

\noindent
{\bf Case 2}: Now consider the case $\p$ ramified in $K$. Set
$K_1:=k(\beta)$, $\beta^p-\beta=f(T)$, $p\nmid \deg f$. Let
$E:=k(w)$ where $w^p-w=\alpha-f(T)=\alpha_1=\sum_{i=1}^r
\frac{Q_i}{P_i^{e_i}}$. 
By Case 1, $E_{{\eu {ge}}}=k(y_1,\ldots,y_r)$. Therefore
$K_{{\eu {ge}}}=E_{{\eu {ge}}}K=k(y_1,\ldots,y_r,\beta)$.

We have proved

\begin{theorem}[S. Hu and Y. Li \cite{HuLi2010}]\label{T5.3.2}
Let $K=k(y)$ be given by
{\rm{(\ref{Eq1})}}. Then $K_{{\eu {ge}}}=k(y_1,\ldots,y_r,\beta)$, where
$y_i^p-y_i=\frac{Q_i}{P^{e_i}}$, $1\leq i\leq r$ and
$\beta^p-\beta=f(T)$.
\end{theorem}

\subsection{$p$--cyclic extensions}\label{S5.3}

This case is similar to Artin--Schreier's. Here we consider
$K=k(\vec{y})$ where $\vec{y}^p\Witt - \vec{y}=\vec {\beta}$, and
the operation is the Witt difference. The extension is a 
finite $p$--extension of degree less than or equal to $p^n$ where
$\vec{y}$ is of length $n$. Let $P_1,\ldots P_r$ be
the finite prime divisors ramified in $K/k$.

\begin{theorem}\label{T5.3.1} Let $K/k$ be a cyclic extension of degree $p^n$
where $P_1,\ldots,P_r\in R_T^+$ and possibly $\p$, are the ramified prime
divisors. Then $K=k(\vec y)$ where
\[
\vec y^p\Witt -\vec y=\vec \beta={\vec\delta}_1\Witt + \cdots \Witt + {\vec\delta}_r
\Witt + \vec\mu,
\]
with $\beta_1^p-\beta_1\notin\wp(k)$,
$\delta_{ij}=\frac{Q_{ij}}{P_i^{e_{ij}}}$, $e_{ij}\geq 0$, $Q_{ij}\in R_T$
and if $e_{ij}>0$, then $p\nmid e_{ij}$, $\gcd(Q_{ij},P_i)=1$ and 
$\deg (Q_{ij})<\deg (P_i^{e_{ij}})$, and $\mu_j=f_j(T)\in R_T$ with
$p\nmid \deg f_j$ when $f_j\not\in {\ma F}_q$.
\end{theorem}

\begin{proof}
We recall some facts on Witt vectors that we will need.
In general, for the ring $R:={\ma Q}[x_i,y_j,z_l]$ in the
variables $x_i,y_j,z_l$ we consider the ring $R_n$, $n\in{\ma N}$,
with the underlying set equal to $R^n$ and
with the operations $+,-,\cdot$ componentwise. Let
$R^n$ be the ring with underlying set the same $R^n$ and with the
following operations (Witt). Let $\varphi \colon R^n\to R_n$ be given by
$\varphi(a_1,\ldots,a_n)=\vWitt a$ where 
\[
a^{(m)}:=a_1^{p^{m-1}}+pa_2^{p^{m-2}}+\cdots+p^{m-1}a_m,
\quad m=1,\ldots,n.
\]
Then $\varphi$ is a bijective map with inverse $\psi\colon
R_n\to R^n$ given by $\psi \vWitt a=(a_1,\ldots,a_n)$
where
\[
a_m=\frac{1}{p^{m-1}}\Big(a^{(m)}-a_1^{p^{m-1}}-pa_2^{p^{m-2}}
-\cdots-p^{m-2}a_{m-1}^p\Big), \quad m=1,\ldots, n.
\]
The Witt operations $\Witt +, \Witt -$ and $\Witt {\cdot}$ on 
$R^n$ are given by
\[
a \Witt {\menos^+_{^{\cdot}}} b=
\Big(a^{\varphi}\menos^+_{^{\cdot}}
b^{\varphi}\Big)^{\varphi^{-1}}.
\]

Now we return to our case of congruence function fields.
Consider $K/k$ a cyclic extension
of degree $p^n$ given by $K:=k(\vec y)$, $\vec y^p\Witt -\vec y=\vec \beta$
with $\vec y \in W_n(K)$ a Witt vector of length $n$ in $K$ and
$\vec\beta\in W_n(k)$ a Witt vector of length $n$ in $k$.

Let $\vec \beta=(\beta_1,\ldots,\beta_n)$ be such that 
\begin{gather}\label{Eq2}
\beta_j=\sum_{i=1}^r \frac{Q_{ij}}{P_i^{e_{ij}}}+f_j(T), \text{\ where\ }
P_1,\ldots,P_r\in R_T^+, \big\{Q_{ij}\big\}_{1\leq i\leq r}^{1\leq j\leq n}
\subseteq R_T,\nonumber\\
f_j(T)\in R_T, e_{ij}\in{\ma N}\cup \{0\} \text{\ for all\ }
1\leq i\leq r \text{\ and\ } 1\leq j\leq n.
\end{gather}

Now when we apply $\varphi$ to $\vec \beta$ we 
obtain $\vWitt \beta$ and from the definition of $\beta^{(j)}$, we obtain
\begin{gather*}
\beta^{(j)}=\sum_{i=1}^r \frac{Q'_{ij}}{P_i^{e'_{ij}}} +f'_j(T)\text{\ for all\ }
1\leq j\leq n.\\
\intertext{We write}
\vec \beta=\vec {\gamma_1}+\cdots+\vec {\gamma_r}+\vec {\xi},\\
\vWitt \beta=\vWitt {\gamma_1}+\cdots+\vWitt {\gamma_r} +
\vWitt {\xi}\\
\intertext{with}
\gamma_i^{(j)}=\frac{Q'_{ij}}{P_i^{e'_{ij}}},\quad 1\leq i\leq r, 1\leq j\leq n
\text{\ and\ } \xi^{(j)}=f'_j(T).
\end{gather*}

When we apply $\varphi^{-1}$, we obtain
\[
(\beta_1,\ldots,\beta_n)=\vWitt {\beta}^{\varphi^{-1}} =
({\vec {\gamma}_1})^{\varphi^{-1}}\Witt +\cdots\Witt + ({\vec {\gamma}_r})^{
\varphi^{-1}}\Witt + ({\vec {\xi}})^{\varphi^{-1}}
\]
and each vector $({\vec {\gamma_i}})^{\varphi^{-1}}$ is of the form
$\Big(\frac{Q''_{i1}}{P_i^{e''_{i1}}},\cdots, \frac{Q''_{in}}{P_i^{e''_{in}}}\Big)$
and the vector $(\vec \xi)^{\varphi^{-1}}$ is of the form $\big(f''_1(T),\ldots,
f''_n(T)\big)$. In other words
\[
\vec \beta={\vec {\delta}}_1\Witt +\cdots\Witt + {\vec {\delta}}_r
\Witt + \vec \mu
\]
where the components of each ${\vec \delta}_i$ have poles at most
at $P_i$ and $\vec \mu$ has components with poles at most at $\p$.
Let ${\eu p}_i$ be the divisor corresponding to $P_i$.

Now each ${\vec\delta}$ and $\vec\mu$ can be normalized in such a 
way that each component $({\vec\delta}_i)_j:=\delta_{ij}$ has divisor
\begin{gather*}
\big(\delta_{ij}\big)_k=\frac{{\eu a}_{ij}}{{\eu p}_i^{\lambda_i}} \text{\ with\ }
\lambda_i\geq 0;
\text{\ if \ }\lambda_i=0, \text{\ then\ }v_{{\eu p}_i}({\eu a}_{ij})\geq 0;\\
\text{\ if\ } \lambda_i>0, \text{\ then\ } \gcd(p,\lambda_i)=1 \text{\ and\ }
v_{{\eu p}_{ij}}({\eu a}_{ij})=0,
\end{gather*}
and similarly for $\vec \mu$ with respect to $\p$ (see \cite[page 162]{Sch36}).
Indeed, the normalization can be obtained by the change of 
variable 
$y_{ij}\mapsto y_{ij}+\alpha_{ij}$, $1\leq i\leq r$, $1\leq j\leq n$,
where ${\vec y}_i=(y_{i1},\ldots, y_{in})$, ${\vec y}_i^p
\Witt - {\vec y}_i ={\vec \delta}_i$ and $\alpha_{ij}\in k$,
that corresponds to the substitution
$\delta_{ij}\mapsto \delta_{ij}+\alpha_{ij}^p-\alpha_{ij}$
and therefore the components obtained
have no poles other than ${\eu p}_i$. 
\end{proof}

Now we study the behavior of $\p$ in $K/k$.

\begin{proposition}\label{P5.3.2}
Let $K/k$ be given as in Theorem {\rm{\ref{T5.3.1}}}. Let 
$\mu_1=\cdots=\mu_s=0$, $\mu_{s+1}\in {\ma F}_q^{\ast}$, 
$\mu_{s+1}\not\in \wp({\ma F}_q)$ and finally, let $t+1$ be the
first index with $f_{t+1}\not\in{\ma F}_q$ (and therefore $p\nmid \deg f_{t+1}$).
Then the ramification index of $\p$ is $p^{n-t}$, the inertia degree
of $\p$ is $p^{t-s}$ and the decomposition number
of $\p$ is $p^s$. More precisely, if $\Gal(K/k)=\langle \sigma\rangle
\cong C_{p^n}$, then the inertia group of $\p$ is ${\eu I}=\langle \sigma^{p^t}
\rangle$ and the decomposition group of $\p$ is ${\eu D}=\langle
\sigma^{p^s}\rangle$.
\end{proposition}

\begin{proof}
Since the extension $K/k$ is a cyclic extension of degree a power
of a prime, the inertia field is the first layer such that $\p$ ramifies.
The index of this first layer is $t+1$ (see \cite{Sch36}). On the
other hand, by the same reason, the decomposition field
is the first layer where $\p$ is inert and this is given by $s+1$
(Proposition \ref{P5.3}).
\end{proof}

Now ${\vec y}_i^p\Witt - {\vec y}_i={\vec\delta}_i$, $1\leq i\leq r$ and
$\vec z^p\Witt -\vec z=\vec\mu$. Note that $k(\vec y,{\vec y}_i)$
and $k(\vec y, \vec z)$ are unramified extensions of $k(\vec y)$.

We have $k(\vec y\Witt -\vec z)\subseteq
 k(\Lambda_N)$ for some $N\in R_T$
since by Proposition \ref{P5.3.2},
$\p$ is fully decomposed in $k(\vec y\Witt - \vec z)$.
Therefore $E=k(\vec y\Witt -\vec z)$ is contained in
a cyclotomic function field.

If $\chi$ is the character associated to $E$, then $\chi=\chi_{P_1}
\cdots\chi_{P_r}$, where each $\chi_{P_i}$ is of order $p^{n_i}$
with $n_i\leq n$. The field associated to $\chi_{P_i}$ is the
field contained in a cyclotomic function field such that $P_i$
is the only ramified prime and with the same ramification behavior
that the one of $P_i$ in $E/k$. Since in both cases this
ramification is completely determined by ${\vec \delta}_i$, it
follows that the field
associated to $\chi_{P_i}$ is $k({\vec y}_i)$. It follows that
$E_{{\eu {ge}}}=k({\vec y}_1,\ldots,{\vec y}_r)$ since $\p$ is fully decomposed.

Note that $Kk(\vec z)/K$ is unramified and $\S K$ decomposes fully.
It follows from Theorems \ref{T4.1} and
\ref{T4.4} that $K_{{\eu {ge}}}=E_{{\eu {ge}}}k(\vec z)$.

Therefore we have proved

\begin{theorem}\label{T5.3.3}
If $K/k$ is given as in Theorem {\rm{\ref{T5.3.1}}}, then
$K_{{\eu {ge}}}=k({\vec y}_1,\ldots,{\vec y}_r,\vec z)$
where ${\vec y}_i^p\Witt -{\vec y}_i ={\vec \delta}_i$, $1\leq i\leq r$
and ${\vec z}^p\Witt -\vec z=\vec\mu$.
\end{theorem}

\begin{example}\label{E5.12}{\rm{
Let $k={\ma F}_3(T)$ and $K=k(\vec y)$ where
 $\vec y^3\Witt - \vec y=\vec {\beta}=
\big(\frac{1}{T}+1,\frac{1}{T+1}+T\big)$. Then the decomposition
prescribed in Theorem \ref{T5.3.1} is:
\[
\vec \beta=\Big(\frac{1}{T},\frac{T+1}{T^2}\Big)\Witt +
\Big(0,\frac{1}{T+1}\Big)\Witt + \big(1,T\big).
\]
Thus, if $\vec y_1^3\Witt -\vec y_1=\vec\delta_1
=\big(\frac{1}{T},\frac{T+1}{T^2}\big)$,
$\vec y_2^3\Witt -\vec y_2=\vec \delta_2=\big(0,\frac{1}{T+1}\big)$ and
$\vec z^3\Witt -\vec z=\vec \mu=\big(1,T\big)$, then
$K_{{\eu {ge}}}=k(\vec y_1,\vec y_2,\vec z)$.
}}
\end{example}

\bibliographystyle{model1-num-names}

\begin{thebibliography}{xx}

\bibitem{BaeKoo96} Bae, Sunghan; Koo, Ja Kyung,
Genus theory for function fields,
J. Austral. Math. Soc. Ser. A 60, no. 3 (1996) 301--310.

\bibitem{Cle92} Clement, Rosario, The genus field of an algebraic function field,
J. Number Theory 40, no. 3 (1992) 359--375.

\bibitem{Fro83} Fr\"ohlich, Albrecht, Central extensions, Galois groups and ideal class
groups of number fields, Contemporary Mathematics, 24, American Mathematical
Society, Providence, RI, 1983.

\bibitem{Gau1801} Gauss, Carl Friedrich, Disquisitiones arithmeticae,
1801.

\bibitem{Has51} Hasse, Helmut, Zur Geschlechtertheorie in quadratischen
Zahlk\"orpern, J. Math. Soc. Japan 3 (1951) 45--51.

\bibitem{Hay74}  Hayes, David R.,
Explicit class field theory for rational function fields,
Trans. Amer. Math. Soc. 189 (1974), 77--91.

\bibitem{HuLi2010} Hu, Su; Li, Yan, The genus fields of Artin--Schreier extensions,
Finite Fields Appl. 16, no. 4  (2010) 255--264.

\bibitem{Ish76} Ishida, Makoto, The genus fields of algebraic number fields,
Lecture Notes in Mathematics, Vol. 555, Springer-Verlag, Berlin-New York, 1976.

\bibitem{Leo53} Leopoldt, Heinrich W., Zur Geschlechtertheorie in abelschen
Zahlk\"orpern, Math. Nachr. 9 (1953) 351--362.

\bibitem{Pen2003} Peng, Guohua, The genus fields of Kummer function fields,
J. Number Theory 98, no. 2 (2003) 221--227.

\bibitem{Ros87} Rosen, Michael, The Hilbert class field in function fields,
Exposition. Math. 5, no. 4 (1987) 365--378.

\bibitem{Ros2002} Rosen, Michael, Number theory in function fields,
Graduate Texts in Mathematics, 210, Springer-Verlag, New York, 2002.

\bibitem{Sch36} Schmid, Hermann Ludwig, Zur Arithmetik
der zyklischen p-K\"orper, J. Reine Angew. Math. 
176 (1936) 161--167.

\bibitem{Vil2006} Villa Salvador, Gabriel Daniel, Topics in the theory of
algebraic function fields, Mathematics: Theory \& Applications. Birkh\"auser Boston,
Inc., Boston, MA, 2006.

\bibitem{Xia85} Zhang, Xianke, A simple construction of genus fields of
abelian number fields, Proc. Amer. Math. Soc. 94, no. 3 (1985) 393--395.

\end{thebibliography}

\end{document}